\documentclass{article}[12pt]
\usepackage{amsfonts,amsmath,amsthm,yfonts,mathdots,
latexsym,graphicx,amssymb,hyperref,url,fancyhdr,lineno,tikz,float}
\usepackage{amsfonts,amsmath,latexsym,hyperref,graphicx,amssymb,amsthm,url,lineno,float}
\usepackage[abbrev,msc-links]{amsrefs}  
\usetikzlibrary{arrows}
\usepackage{rotating}
\usetikzlibrary{calc}
\usetikzlibrary{positioning}
\usetikzlibrary{patterns}
\usetikzlibrary{shapes}
\usepackage{subcaption}

\title{A result in asymmetric Euclidean Ramsey theory}

\author{Andrii Arman\thanks{University of Manitoba, \url{armana@myumanitoba.ca}}
\and Sergei Tsaturian\thanks{University of Manitoba, \url{tsaturis@cc.umanitoba.ca}}}
\date{15 February 2017}

\begin{document}
\maketitle

\begin{abstract}
It is proved that if the points of the three-dimensional Euclidean space are coloured in red and blue, then there exist either two red points unit distance apart, or six collinear blue points with distance one between any two consecutive points.
\end{abstract}

\newtheorem{theorem}{Theorem}[section]
\newtheorem{lemma}[theorem]{Lemma}
\newtheorem{corollary}[theorem]{Corollary}
\newtheorem{conjecture}{Conjecture}
\newtheorem{definition}[theorem]{Definition}
\newtheorem{problem} [theorem]{Problem}
\newtheorem{proposition} [theorem]{Proposition}
\newtheorem{example}[theorem]{Example}
\newtheorem{question}{Question}


\section{Introduction}
The area of Euclidean Ramsey theory is mostly concerned with colouring points of an Euclidean space $\mathbb{E}^n$ and looking for a geometric graph whose points are monochromatic. Research in this area was surveyed in \cites{erdos1, erdos2, erdos3} by Erd\H os, Graham, Montgomery, Rothschild, Spencer,
and Straus.\\
Asymmetric Euclidean Ramsey theory deals with the questions of the following type. If $F_1$ and $F_2$ are two different configurations of points in an Euclidean space, is it true that for any colouring of the points of an Euclidean space $\mathbb{E}^n$ in red and blue, there always exists either a congruent copy of $F_1$ with all of the vertices red, or a congruent copy of $F_2$ with all of the vertices blue? In the case of affirmative answer, the ``Ramsey arrow`` notation is used:
$$\mathbb{E}^n\rightarrow (F_1,F_2).$$
The results of Erd\H os et al. \cites{erdos1,erdos2,erdos3} include:
\begin{itemize}
\item If $T_i$ is a configuration of $i$ points, then $\mathbb{E}^2\rightarrow (T_2,T_3);$
\item If $T$ denotes the isosceles right triangle with unit catheti (legs) and $Q^2$ denotes the unit square, then $\mathbb{E}^3\rightarrow (T,Q^2);$
\item Let $\ell_i$ denote the configuration of $i$ collinear points with distance $1$ between any two consecutive points. Then $\mathbb{E}^2\rightarrow (\ell_2,\ell_4)$ and $\mathbb{E}^4\rightarrow (\ell_2,\ell_5)$.
\end{itemize}
Juh\'asz~\cite{juhasz} proved that if $T_4$ is any configuration of $4$ points, then $\mathbb{E}^2\rightarrow (\ell_2,T_4)$. 
Juh\'asz (personal
communication, 10 February, 2017) also informed us that Iv\'an's master thesis~\cite{ivan} contains a proof that for any configuration $T_5$ of $5$ points, $\mathbb{E}^3\rightarrow (\ell_2,T_5)$, but this result was never published.


It was asked by Erd\H os et al. \cite{erdos2} whether $\mathbb{E}^3\rightarrow (\ell_2,\ell_5)$. The result of Iv\'an~\cite{ivan} implies the positive answer to this question. We present a simple proof of $\mathbb{E}^3\rightarrow (\ell_2,\ell_5)$ and prove a stronger result, namely that $\mathbb{E}^3\rightarrow (\ell_2,\ell_6)$.\\
For an overview of other results in Euclidean Ramsey theory, see Graham's survey \cite{graham}.

\section{Existence of a blue $\ell_5$}
\begin{theorem}\label{thm:l5}
Let the Euclidean space $\mathbb{E}^3$ be coloured in red and blue so that there are no two red points distance $1$ apart. Then there exist five blue points that form an $\ell_5$.
\end{theorem}
To prove Theorem~\ref{thm:l5}, it is first proved (in the following two lemmas) that if there is a colouring that forbids red $\ell_2$ and blue $\ell_5$, there are no two red points distance $2$ or $\sqrt{7}$ apart.
\begin{lemma}\label{l:l52}
Let $\mathbb{E}^3$ be coloured in red and blue so that there are no two red points distance $1$ apart. If there are no five blue points forming an $\ell_5$, then there are no two red points distance $2$ apart.
\end{lemma}
\begin{proof}
Assume that points $A$ and $B$ are red and $|AB|=2$. Choose a rectangular coordinate system centered at $A$ so that $B$ has coordinates $(2,0,0)$. Then any point at distance $1$ to $A$ is blue, otherwise there are two red points distance $1$ apart. In particular, the circles $\{(-\frac12,y,z):y^2+z^2=\frac34\}$ and $\{(\frac12,y,z):y^2+z^2=\frac34\}$ are blue (see Figure~\ref{fig:1}). Similarly, since any point at distance $1$ to $B$ is blue, the circles $\{(\frac32,y,z):y^2+z^2=\frac34\}$ and $\{(\frac52,y,z):y^2+z^2=\frac34 \}$ are blue. Then, for any point $P_1$ on the first circle, the line $\ell$ through $P_1$ parallel to $AB$ intersects the other three circles at some points $P_2, P_3, P_4$ that together with $P_1$ form a blue $\ell_4$. Therefore the point $P_5$ of intersection of $\ell$ with the circle $C=\{(\frac72,y,z):y^2+z^2=\frac34\}$ is red (otherwise $P_1, P_2, P_3, P_4, P_5$ form a blue $\ell_5$). Since $P_1$ can be any point on $\{(-\frac12,y,z):y^2+z^2=\frac34\}$, the whole circle $C$ is red. Since the radius of $C$ is $\frac{\sqrt{3}}{2}$, $C$ contains two red points distance $1$ apart, which contradicts the assumptions of the lemma.
\end{proof}

\begin{figure}[H]
\begin{center}
\begin{tikzpicture}[line cap=round,line join=round,>=triangle 45,x=1.0cm,y=1.0cm, scale=1.5]
\tikzstyle{dc}   = [circle, minimum width=6pt, draw, inner sep=0pt, path picture={\draw (path picture bounding box.south east) -- (path picture bounding box.north west) (path picture bounding box.south west) -- (path picture bounding box.north east);}]
\clip(-1,-1.3) rectangle (7,1);
\draw [color=black, opacity=0.5] (0,0)--(5.5,0);
\draw [color=black] (-0.5, 0)++(-30:0.2cm and 0.866cm)--+(5.5,0);
\draw[color=black, dashed] (0.5, 0)+(-30:0.2cm and 0.866cm)--(1,0);
\draw[color=black, dashed] (1.5, 0)+(-30:0.2cm and 0.866cm)--(1,0);
\draw[color=black, dashed] (2.5, 0)+(-30:0.2cm and 0.866cm)--(3,0);
\draw[color=black, dashed] (3.5, 0)+(-30:0.2cm and 0.866cm)--(3,0);

\node [above] at (1,0.1) {{$A$}};
\node [above] at (3,0.1) {{$B$}};


\foreach \a in {0,1,2,3}
\draw [color=blue, dashed] (0.5+\a,-0.866) arc (-90:-270:0.2 and 0.866);

\foreach \a in {0,1,2,3}
\draw [color=blue] (0.5+\a,-0.866) arc (-90:90:0.2 and 0.866);

\draw [color=red, dashed] (4.5,-0.866) arc (-90:-270:0.2 and 0.866);
\draw [color=red] (4.5,-0.866) arc (-90:90:0.2 and 0.866);

\draw[fill=black] (4.5,0)+(20:0.35cm and 1.266cm) circle (0pt) node {$C$};
\draw  (4.5, 0)++(-30:0.2cm and 0.866cm) node[diamond, scale=0.5, fill=red] {};
\foreach \a in {1,2,3,4,5}
\draw[fill] (-0.5+\a, 0)++(-30:0.4cm and 1.266cm) circle (0pt) node {$P_{\a}$};
\foreach \a in {0,1,2,3}
\draw [fill=blue] (0.5+\a, 0)+(-30:0.2cm and 0.866cm) circle (2pt);
\draw (1,0) node[diamond, scale=0.5, fill=red] {};
\node[diamond, scale=0.5, fill=red, scale=1] at (3,0) {};

\node[diamond, fill=red, scale=0.5] at (0,-1.2) {};
\draw[fill=red] (0.8,-1.2) circle (0pt) node { -- red point};

\draw [fill=blue] (2,-1.2) circle (2pt);
\draw[fill=red] (2.8,-1.2) circle (0pt) node { -- blue point};

\end{tikzpicture}
\end{center}
\caption{}\label{fig:1}
\end{figure}

\begin{lemma}\label{l:l5r7}
Let $\mathbb{E}^3$ be coloured in red and blue so that there are no two red points distance $1$ apart. If there are no five blue points forming an $\ell_5$, then there are no two red points distance $\sqrt{7}$ apart.
\end{lemma}
\begin{proof}
Assume that points $A$ and $B$ are red and $|AB|=\sqrt{7}$. Choose a rectangular coordinate system so that $A$ has coordinates $(-\frac{\sqrt{7}}{2},0,0)$ and $B$ has coordinates $(\frac{\sqrt{7}}{2},0,0)$. Since $A$ and $B$ are $\sqrt{7}$ apart, there exist points $P_1,P_2,P_3,P_4$, all collinear, such that $|P_1P_2|=|P_2P_3|=|P_3P_4|=1$ and the triangles $AP_1P_2$ and $BP_3P_4$ are equilateral (see Figure~\ref{fig:2}). Let $P_1,P_2,P_3,P_4$ have coordinates $P_1(-\frac{3}{\sqrt{7}},\frac{3\sqrt{3}}{2\sqrt{7}},0)$, $P_2(-\frac{1}{\sqrt{7}},\frac{\sqrt{3}}{2\sqrt{7}},0)$, $P_3(\frac{1}{\sqrt{7}},-\frac{\sqrt{3}}{2\sqrt{7}},0)$, $P_4(\frac{3}{\sqrt{7}},-\frac{3\sqrt{3}}{2\sqrt{7}},0)$). Then $P_1, P_2, P_3, P_4$ form a blue $\ell_4$, therefore the point $P_5(\frac{5}{\sqrt{7}}, -\frac{5\sqrt{3}}{2\sqrt{7}},0)$ is red. When the line $P_1P_2P_3P_4$ is rotated about $AB$, points $P_1, P_2, P_3, P_4$ span four blue circles (since every point on the circles is at distance $1$ to either $A$ or $B$), therefore $P_5$ (when rotated) spans a red circle with radius $\frac{5\sqrt{3}}{2\sqrt{7}}>1$ that contains two red points distance $1$ apart.
\end{proof}

\begin{figure}[H]
\begin{center}
\begin{tikzpicture}[line cap=round,line join=round,>=triangle 45,x=1.0cm,y=1.0cm, scale=1]

\clip(-4.5,-2.2) rectangle (7,1.7);
\begin{scope}[xshift=-3cm]
\draw [color=black, opacity=0.5] (0, 0)--(2.6457,0);
\draw [color=black] (3.2126,-1.6365)--(0.189,0.9819);
\draw [color=black] (0, 0)--(0.189,0.9819);
\draw[color=black] (0, 0)--(0.9449,0.3273);
\draw[color=black] (2.6457,0)--(1.7008,-0.3273);
\draw[color=black] (2.6457,0)--(2.4567,-0.9819);
\draw [fill=red] (0.,0.) node[diamond, scale=0.5, fill=red] {};
\node [left] at (0,0) {{$A$}};

\draw [fill=red] (2.6457,0) node[diamond, scale=0.5, fill=red] {};
\draw[fill=black] (2.6457+0.2,0) circle (0pt) node {$B$};

\draw [fill=blue] (0.189,0.9819) circle (2pt);
\draw [fill=black] (0.189,1.2819) circle (0pt) node {$P_1$};

\draw [fill=blue] (0.9449,0.3273) circle (2pt);
\draw [fill=black] (0.9449,0.6273) circle (0pt) node {$P_2$};

\draw [fill=blue] (1.7008,-0.3273) circle (2pt);
\draw [fill=black] (1.7008,-0.6273) circle (0pt) node {$P_3$};

\draw [fill=blue] (2.4567,-0.9819) circle (2pt);
\draw [fill=blue] (2.4567,-1.2819) circle (0pt) node {$P_4$};

\draw [fill=red] (3.2126,-1.6365) node[diamond, scale=0.5, fill=red] {};
\draw [fill=blue] (3.2126,-1.9365) circle (0pt) node {$P_5$};
\end{scope}

\begin{scope}[xshift=2cm]
\draw [color=black, opacity=0.5] (0, 0)--(2.6457,0);
\draw [color=black] (3.2126,-1.6365)--(0.189,0.9819);
\draw [color=black] (0, 0)--(0.189,0.9819);
\draw[color=black] (0, 0)--(0.9449,0.3273);
\draw[color=black] (2.6457,0)--(1.7008,-0.3273);
\draw[color=black] (2.6457,0)--(2.4567,-0.9819);

\draw [color=blue, dashed] (0.189,0.9819) arc (90:270:0.3 and 0.9819);
\draw [color=blue] (0.189,0.9819) arc (90:-90:0.3 and 0.9819);

\draw [color=blue, dashed] (0.9449,0.3273) arc (90:270:0.1 and 0.3273);
\draw [color=blue] (0.9449,0.3273) arc (90:-90:0.1 and 0.3273);

\draw [color=blue, dashed] (1.7008,0.3273) arc (90:270:0.1 and 0.3273);
\draw [color=blue] (1.7008,0.3273) arc (90:-90:0.1 and 0.3273);

\draw [color=blue, dashed] (2.4567,0.9819) arc (90:270:0.3 and 0.9819);
\draw [color=blue] (2.4567,0.9819) arc (90:-90:0.3 and 0.9819);

\draw [color=red, dashed] (3.2126,-1.6365) arc (-90:-270:0.4 and 1.6365);
\draw [color=red] (3.2126,-1.6365) arc (-90:90:0.4 and 1.6365);

\draw [fill=red] (0.,0.) node[diamond, scale=0.5, fill=red] {};
\node [left] at (0,0) {{$A$}};
\draw [fill=red] (2.6457,0) node[diamond, scale=0.5, fill=red] {};
\draw[fill=black] (2.6457+0.2,0) circle (0pt) node {$B$};

\draw [fill=blue] (0.189,0.9819) circle (2pt);
\draw [fill=black] (0.189,1.2819) circle (0pt) node {$P_1$};

\draw [fill=blue] (0.9449,0.3273) circle (2pt);
\draw [fill=black] (0.9449,0.6273) circle (0pt) node {$P_2$};

\draw [fill=blue] (1.7008,-0.3273) circle (2pt);
\draw [fill=black] (1.7008,-0.6273) circle (0pt) node {$P_3$};

\draw [fill=blue] (2.4567,-0.9819) circle (2pt);
\draw [fill=blue] (2.4567,-1.2819) circle (0pt) node {$P_4$};

\draw [fill=red] (3.2126,-1.6365) node[diamond, scale=0.5, fill=red] {};
\draw [fill=blue] (3.2126,-1.9365) circle (0pt) node {$P_5$};

\end{scope}

\end{tikzpicture}
\end{center}
\caption{}\label{fig:2}
\end{figure}

\begin{proof}[Proof of Theorem~\ref{thm:l5}]
Let $\mathbb{E}^3$ be coloured in red and blue in a way that there is no red $\ell_2$ or blue $\ell_5$. Consider any red point $A$ and a rectangular coordinate system centered at $A$. By Lemma~\ref{l:l52}, any point at distance $2$ to $A$ is blue, in particular, the circles $\{(1,y,z):y^2+z^2=3\}$ and $\{(-1,y,z):y^2+z^2=3\}$ are blue (see Figure~\ref{fig:3}). Similarly, by Lemma~\ref{l:l5r7}, the circles $\{(2,y,z):y^2+z^2=3\}$ and $\{(-2,y,z):y^2+z^2=3\}$ are blue. Consider any point $P_3$ on the circle $\{(0,y,z):y^2+z^2=3\}$. The line through $P_3$ parallel to $y=z=0$ intersects the four blue circles at points $P_1, P_2, P_4$ and $P_5$ that together with $P_3$ form an $\ell_5$. Since $P_1, P_2, P_4, P_5$ are blue, $P_3$ is red. Therefore, the circle $\{(0,y,z):y^2+z^2=3\}$ is red and has radius $\sqrt{3}$, and so this circle contains two red points distance $1$ apart.
\end{proof}

\begin{figure}[H]
\begin{center}
\begin{tikzpicture}[line cap=round,line join=round,>=triangle 45,x=1.0cm,y=1.0cm, scale=1]
\tikzstyle{dc}   = [circle, minimum width=6pt, draw, inner sep=0pt, path picture={\draw (path picture bounding box.south east) -- (path picture bounding box.north west) (path picture bounding box.south west) -- (path picture bounding box.north east);}]

\clip(-4.5,-2) rectangle (7.5,2.5);
\begin{scope}[xshift=-1.5cm]
\draw [color=black, opacity=0.5] (-2,0)--(2, 0);
\draw [color=black, dashed] (0,0)++(60:2)--(0, 0);
\draw [color=black, dashed] (0,0)++(120:2)--(0, 0);
\draw [color=black, dashed] (-1,0)++(120:2)--+(4, 0);
\draw [color=black, dashed] (0,0)++(60:1)--+(-1, 0);
\draw [color=black, dashed] (0,0)++(60:1)++(120:1)--+(-60: 1);
\draw [color=black, dashed] (0,0)++(60:1)++(120:1)--+(-120: 1);

\draw [fill=red] (0.,0.) node[diamond, scale=0.5, fill=red] {};
\node [below] at (0,-0.1) {{$A$}};

\draw [fill=blue] (-1,0)++(120:2) circle (2pt);
\draw [fill=black] (-1,0.3)++(120:2) circle (0pt) node {$P_1$};

\draw [fill=blue] (0,0)++(120:2) circle (2pt);
\draw [fill=black] (0,0.3)++(120:2) circle (0pt) node {$P_2$};

\draw [fill=red] (0,0)++(60:1)++(120:1) node[diamond, scale=0.5, fill=red] {};
\draw [fill=black] (0,0.3)++(60:1)++(120:1) circle (0pt) node {$P_3$};

\draw [fill=blue] (0,0)++(60:2) circle (2pt);
\draw [fill=blue] (0,0.3)++(60:2) circle (0pt) node {$P_4$};

\draw [fill=blue] (1,0)++(60:2) circle (2pt);
\draw [fill=blue] (1,0.3)++(60:2) circle (0pt) node {$P_5$};
\end{scope}

\begin{scope}[xshift=1cm]

\end{scope}

\begin{scope}[xshift=4.5cm]
\draw [color=black, opacity=0.5] (-2,0)--(2, 0);
\draw [color=black, dashed] (-1,0)++(120:2)--+(4, 0);

\foreach \a in {1,2,4,5}
\draw [color=blue, dashed] (-3+\a,0)++(60:1)++(120:1) arc (90:270:0.4 and 1.732);

\foreach \a in {1,2,4,5}
\draw [color=blue] (-3+\a,0)++(60:1)++(120:1) arc (90:-90:0.4 and 1.732);

\draw [color=red, dashed] (0,0)++(60:1)++(120:1) arc (90:270:0.4 and 1.732);
\draw [color=red] (60:1)++(120:1) arc (90:-90:0.4 and 1.732);


\draw [fill=red] (0.,0.) node[diamond, scale=0.5, fill=red] {};
\node [below] at (0,-0.1) {{$A$}};

\draw [fill=blue] (-1,0)++(120:2) circle (2pt);
\draw [fill=black] (-1,0.3)++(120:2) circle (0pt) node {$P_1$};

\draw [fill=blue] (0,0)++(120:2) circle (2pt);
\draw [fill=black] (0,0.3)++(120:2) circle (0pt) node {$P_2$};

\draw [fill=red] (0,0)++(60:1)++(120:1) node[diamond, scale=0.5, fill=red] {};
\draw [fill=black] (0,0.3)++(60:1)++(120:1) circle (0pt) node {$P_3$};

\draw [fill=blue] (0,0)++(60:2) circle (2pt);
\draw [fill=blue] (0,0.3)++(60:2) circle (0pt) node {$P_4$};

\draw [fill=blue] (1,0)++(60:2) circle (2pt);
\draw [fill=blue] (1,0.3)++(60:2) circle (0pt) node {$P_5$};

\end{scope}

\end{tikzpicture}
\caption{}\label{fig:3}
\end{center}
\end{figure}

\section{Existence of a blue $\ell_6$}

\begin{theorem}\label{thm:l6}
Let the Euclidean space $\mathbb{E}^3$ be coloured in red and blue so that there are no two red points distance $1$ apart. Then there exist six blue points that form an $\ell_6$.
\end{theorem}

For the sake of contradiction, it is assumed that there exists a colouring of $\mathbb{E}^3$ in red and blue without red $\ell_2$ and blue $\ell_6$. The following four lemmas are needed.

\begin{lemma}\label{l:disk}
Let $\mathbb{E}^3$ be coloured in red and blue so that there are no two red points distance $1$ apart. If there are no six blue points forming an $\ell_6$, then there is no disk with radius $\sqrt{3}$, such that all of its points (including interior) are blue.
\end{lemma}
\begin{proof}
Suppose that there exists a blue disk $D$ with center $O$ and radius $\sqrt{3}$. Consider a rectangular coordinate system on the plane containing $D$, centered at $O$ (then $D=\{x^2+y^2\leq 3, x,y\in\mathbb{R}\}$). Let $P_4$ be any point on the boundary of $D$ (for simplicity, let the coordinates of $P_4$ be $(\sqrt{3},0)$. Then points $P_4$, $P_3(\sqrt{3}-1,0)$, $P_2(\sqrt{3}-2,0)$, $P_1(\sqrt{3}-3,0)$ belong $D$, and therefore are blue (see Figure~\ref{fig:4a}). Consider points $P_5(\sqrt{3}+1,0)$, $P_6(\sqrt{3}+2,0)$, and a point $A$ (say, $(\sqrt{3}+\frac32, \frac{\sqrt{3}}{2})$) at distance $1$ to both $P_5$ and $P_6$. If $A$ is red, then both $P_5$ and $P_6$ are blue and $P_1,P_2,P_3,P_4,P_5,P_6$ form a blue $\ell_6$. Therefore $A$ is blue. When the point $P_4$ is rotated around the center, the point $A$ (when rotated) spans a blue circle $C$ with radius $\sqrt{(\sqrt{3}+\frac32)^2+(\frac{\sqrt{3}}{2})^2}=\sqrt{6+3\sqrt{3}}$.

Consider any point $Q_6$ on $C$ (for simplicity, let the coordinates of $Q_6$ be $(\sqrt{6+3\sqrt{3}},0)$; see Figure~\ref{fig:4b}). Since $-\sqrt{3}<\sqrt{6+3\sqrt{3}}-5$ and $\sqrt{6+3\sqrt{3}}-2<\sqrt{3}$, points $Q_4(\sqrt{6+3\sqrt{3}}-2,0)$, $Q_3(\sqrt{6+3\sqrt{3}}-3,0)$, $Q_2(\sqrt{6+3\sqrt{3}}-4,0)$, $Q_1(\sqrt{6+3\sqrt{3}}-5,0)$ are all inside $D$, therefore blue. Then the point $Q_5(\sqrt{6+3\sqrt{3}}-1,0)$ is red (otherwise $Q_1Q_2Q_3Q_4Q_5Q_6$ is a blue $\ell_6$). If the point $Q_6$ is chosen arbitrarily on $C$, $Q_5$ spans a red circle with radius $\sqrt{6+3\sqrt{3}}-1>1$, that contains two red points distance 1 apart.
\end{proof}

\begin{figure}[H]
    \begin{subfigure}{0.5\linewidth}
        
\begin{tikzpicture}[line cap=round,line join=round,>=triangle 45,x=1.0cm,y=1.0cm, scale=0.75]

\clip(-3.4,-3.5) rectangle (4.1,3.5);

\draw [color=blue, pattern=north east lines, pattern color=blue] (0,0) circle (1.732cm);

\draw [fill=blue] (-1.2675,0) circle (2pt);
\draw [fill=blue] (-1.2675,0)++(0.1,0.3) circle (0pt) node {$P_1$};

\draw [fill=blue] (-0.2675,0) circle (2pt);
\draw [fill=blue] (-0.2675,0)++(0.1,0.3) circle (0pt) node {$P_2$};

\draw [fill=blue] (0.7325,0) circle (2pt);
\draw [fill=blue] (0.7325,0)++(0.1,0.3) circle (0pt) node {$P_3$};

\draw [fill=blue] (1.7325,0) circle (2pt);
\draw [fill=blue] (1.7325,0)++(0.1,0.3) circle (0pt) node {$P_4$};

\draw [fill=white] (2.7325,0) circle (2pt);
\draw [fill=blue] (2.7325,0)++(0.1,0.3) circle (0pt) node {$P_5$};

\draw [fill=white] (3.7325,0) circle (2pt);
\draw [fill=blue] (3.7325,0)++(0.1,0.3) circle (0pt) node {$P_6$};

\draw [fill=blue] (3.7325,0)++(120:1) circle (2pt);
\draw [fill=blue] (3.7325,0)++(120:1)++(0.1,0.3) circle (0pt) node {$A$};

\draw [color=blue] (0,0) circle (3.346cm);
\draw [fill=blue] (-30:3.346)++(0.4,-0.3) circle (0pt) node {$C$};

\draw [fill=blue] (-140:2)++(0.3,-0.3) circle (0pt) node {$D$};
            
\end{tikzpicture}
        \caption{}
        \label{fig:4a}
    \end{subfigure}%
    \begin{subfigure}{0.5\linewidth}
        \centering
\begin{tikzpicture}[line cap=round,line join=round,>=triangle 45,x=1.0cm,y=1.0cm, scale=0.75]
\clip(-3.4,-3.5) rectangle (4,3.5);

\draw [color=blue, pattern=north east lines, pattern color=blue] (0,0) circle (1.732cm);

\draw [fill=blue] (-1.654,0) circle (2pt);
\draw [fill=blue] (-1.654,0.3)++(0.2,0) circle (0pt) node {$Q_1$};

\draw [fill=blue] (-0.654,0) circle (2pt);
\draw [fill=blue] (-0.654,0.3)++(0.2,0) circle (0pt) node {$Q_2$};

\draw [fill=blue] (0.346,0) circle (2pt);
\draw [fill=blue] (0.346,0.3)++(0.2,0) circle (0pt) node {$Q_3$};

\draw [fill=blue] (1.346,0) circle (2pt);
\draw [fill=blue] (1.346,0.3)++(0.2,0) circle (0pt) node {$Q_4$};

\draw [fill=red] (2.346,0) node[diamond, scale=0.5, fill=red] {};
\draw [fill=blue] (2.346,0.3)++(0.2,0) circle (0pt) node {$Q_5$};

\draw [fill=blue] (3.346,0) circle (2pt);
\draw [fill=blue] (3.346,0.3)++(0.2,0) circle (0pt) node {$Q_6$};

\draw [color=blue] (0,0) circle (3.346cm);
\draw [color=red] (0,0) circle (2.346cm);

\draw [fill=blue] (-30:3.346)++(0.4,-0.3) circle (0pt) node {$C$};

\draw [fill=blue] (-140:2)++(0.3,-0.3) circle (0pt) node {$D$};

\end{tikzpicture}
        \caption{}
        \label{fig:4b}
    \end{subfigure}
    \caption{} 
    \label{fig:4}
\end{figure}

\begin{lemma}\label{l:l62}
Let $\mathbb{E}^3$ be coloured in red and blue so that there are no two red points distance $1$ apart. If there are no six blue points forming an $\ell_6$, then there are no two red points distance $2$ apart.
\end{lemma}
\begin{proof}
The setup is the same as in Lemma~\ref{l:l52}. Assume that points $A$ and $B$ are red and $|AB|=2$. Choose a rectangular coordinate system centered at $A$ so that $B$ has coordinates $(2,0,0)$. Then the circles $\{(-\frac12,y,z):y^2+z^2=\frac34\}$, $\{(\frac12,y,z):y^2+z^2=\frac34\}$, $\{(\frac32,y,z):y^2+z^2=\frac34\}$ and $\{(\frac52,y,z):y^2+z^2=\frac34\}$ are blue. Consider any line $\ell$ parallel to $AB$ that intersects all the blue circles (for simplicity, let $\ell=\{(x,-\frac{\sqrt{3}}{2},0), x\in\mathbb{R}\}$), and let $P_1, P_2, P_3, P_4$ be the points of intersection (see Figure~\ref{fig:5}, in this case $P_1(-\frac12,-\frac{\sqrt{3}}{2},0)$, $P_2(\frac12,-\frac{\sqrt{3}}{2},0)$, $P_3(\frac32,-\frac{\sqrt{3}}{2},0)$, $P_4(\frac52,-\frac{\sqrt{3}}{2},0)$). Let $X$ be any point on the circle $C$ with radius $\frac{\sqrt{3}}{2}$ centered at the point on $\ell$ with $x$-coordinate $4$ (in Figure~\ref{fig:5}, the center has coordinates $(4,-\frac{\sqrt{3}}{2},0)$). If $X$ is red, then the points $P_5$ and $P_6$ on $\ell$ with $x$-coordinates $\frac72$ and $\frac92$ (on the Figure~\ref{fig:5}, $P_5(\frac72,-\frac{\sqrt{3}}{2},0)$,  $P_6(\frac92,-\frac{\sqrt{3}}{2},0)$) are both at distance $1$ to $X$. Then $P_5$ and $P_6$ are blue, and $P_1P_2P_3P_4P_5P_6$ is a blue $\ell_6$. Therefore, $X$ is blue, hence the circle $C$ is blue. When $\ell$ is rotated around $AB$, the circle $C$ spans a blue disk with radius $\sqrt{3}$, which contradicts the statement of Lemma~\ref{l:disk}.
\end{proof}

\begin{figure}[H]
\begin{center}
\begin{tikzpicture}[line cap=round,line join=round,>=triangle 45,x=1.0cm,y=1.0cm, scale=1]
\tikzstyle{dc}   = [circle, minimum width=6pt, draw, inner sep=0pt, path picture={\draw (path picture bounding box.south east) -- (path picture bounding box.north west) (path picture bounding box.south west) -- (path picture bounding box.north east);}]

\clip(-4.5,-2.5) rectangle (7.5,2.5);
\begin{scope}[xshift=-3.5cm]
\draw [color=black, opacity=0.5] (0,0)--(4, 0);
\draw [color=black, opacity=0.5] (-1,0)++(-120:1)--+(6.2, 0);
\draw (-0.9,-0.5) circle (0pt) node{$\ell$};

\draw [fill=red] (0.,0.) node[diamond, scale=0.5, fill=red] {};
\node [above] at (0,0.1) {{$A$}};

\draw [fill=blue] (0,0)++(-120:1) circle (2pt);
\draw [fill=black] (0,-0.3)++(-120:1) circle (0pt) node {$P_1$};

\draw [fill=blue] (0,0)++(-60:1) circle (2pt);
\draw [fill=black] (0,-0.3)++(-60:1) circle (0pt) node {$P_2$};

\draw [fill=red] (2,0.) node[diamond, scale=0.5, fill=red] {};
\node [above] at (2,0.1) {{$B$}};

\draw [fill=blue] (2,0)++(-120:1) circle (2pt);
\draw [fill=black] (2,-0.3)++(-120:1) circle (0pt) node {$P_3$};

\draw [fill=blue] (2,0)++(-60:1) circle (2pt);
\draw [fill=black] (2,-0.3)++(-60:1) circle (0pt) node {$P_4$};

\draw [fill=blue] (3.5,0)++(-60:1)++(30:0.2 and 0.865) circle (2pt);
\draw [fill=black] (3.5,0)++(-60:1)++(30:0.2 and 0.865)++(0.3,0.2) circle (0pt) node {$X$};

\draw [fill=white] (4,0)++(-120:1) circle (2pt);
\draw [fill=black] (4,-0.3)++(-120:1) circle (0pt) node {$P_5$};

\draw [fill=white] (4,0)++(-60:1) circle (2pt);
\draw [fill=black] (4,-0.3)++(-60:1) circle (0pt) node {$P_6$};

\draw [color=blue, dashed] (4,0) arc (90:270:0.2 and 0.865);
\draw [color=blue] (4,0) arc (90:-90:0.2 and 0.865);

\draw [fill=black] (3.5,0)++(-60:1)++(270:0.2 and 0.865)++(0,-0.2) circle (0pt) node {$C$};

\end{scope}

\begin{scope}[xshift=2.5cm]
\draw [color=black, opacity=0.5] (0,0)--(4, 0);
\draw [color=black, opacity=0.5] (0,0)++(-120:1)--+(5, 0);

\draw [fill=red] (0.,0.) node[diamond, scale=0.5, fill=red] {};
\node [above] at (0,0.1) {{$A$}};

\draw [fill=blue] (0,0)++(-120:1) circle (2pt);
\draw [fill=black] (0,-0.3)++(-120:1) circle (0pt) node {$P_1$};

\draw [fill=blue] (0,0)++(-60:1) circle (2pt);
\draw [fill=black] (0,-0.3)++(-60:1) circle (0pt) node {$P_2$};

\draw [fill=red] (2,0.) node[diamond, scale=0.5, fill=red] {};
\node [above] at (2,0.1) {{$B$}};

\draw [fill=blue] (2,0)++(-120:1) circle (2pt);
\draw [fill=black] (2,-0.3)++(-120:1) circle (0pt) node {$P_3$};

\draw [fill=blue] (2,0)++(-60:1) circle (2pt);
\draw [fill=black] (2,-0.3)++(-60:1) circle (0pt) node {$P_4$};


\draw [fill=white] (4,0)++(-120:1) circle (2pt);
\draw [fill=black] (4,-0.3)++(-120:1) circle (0pt) node {$P_5$};

\draw [fill=white] (4,0)++(-60:1) circle (2pt);
\draw [fill=black] (4,-0.3)++(-60:1) circle (0pt) node {$P_6$};

\draw [color=blue, dashed] (4,0)++(60:1)++(120:1) arc (90:270:0.3 and 1.732);
\draw [color=blue] (4,0)++(60:1)++(120:1) arc (90:-90:0.3 and 1.732);

\draw[color=white, pattern=north east lines, pattern color=blue] (4,0) circle (0.3 and 1.732);

\draw [color=blue, dashed] (4,1.732) arc (90:270:0.3 and 1.732);
\draw [color=blue] (4,1.732) arc (90:-90:0.3 and 1.732);

\end{scope}

\end{tikzpicture}
\caption{}\label{fig:5}
\end{center}
\end{figure}

\begin{lemma}\label{l:l64}
Let $\mathbb{E}^3$ be coloured in red and blue so that there are no two red points distance $1$ apart. If there are no six blue points forming an $\ell_6$, then there are no two red points distance $4$ apart.
\end{lemma}
\begin{proof}
The proof is similar to that of Lemma~\ref{l:l62}. Assume that points $A$ and $B$ are red and $|AB|=4$. Choose a rectangular coordinate system centered at $A$ so that $B$ has coordinates $(4,0,0)$. Then the circles $\{(-\frac12,y,z):y^2+z^2=\frac34\}$, $\{(\frac12,y,z):y^2+z^2=\frac34\}$, $\{(\frac72,y,z):y^2+z^2=\frac34\}$ and $\{(\frac92,y,z):y^2+z^2=\frac34\}$ are blue. Consider any line $\ell$ parallel to $AB$ that intersects all the blue circles (for simplicity, let $\ell=\{(x,-\frac{\sqrt{3}}{2},0): x\in \mathbb{R}\}$), and let $P_1, P_2, P_5, P_6$ be the points of intersection (see Figure~\ref{fig:6}, in this case $P_1(-\frac12,-\frac{\sqrt{3}}{2},0)$, $P_2(\frac12,-\frac{\sqrt{3}}{2},0)$, $P_5(\frac72,-\frac{\sqrt{3}}{2},0)$, $P_6(\frac92,-\frac{\sqrt{3}}{2},0)$). Let $X$ be any point on the circle $C$ with radius $\frac{\sqrt{3}}{2}$ centered at the point on $\ell$ with $x$-coordinate $2$ (on the Figure~\ref{fig:6}, the center has coordinates $(2,-\frac{\sqrt{3}}{2},0)$). If $X$ is red, then the points $P_3$ and $P_4$ on $\ell$ with $x$-coordinates $\frac32$ and $\frac52$  (on the Figure~\ref{fig:6} $P_3(\frac32,-\frac{\sqrt{3}}{2},0)$, $P_4(\frac52,-\frac{\sqrt{3}}{2},0)$) are both at distance $1$ to $X$. Then $P_3$ and $P_4$ are blue, and $P_1P_2P_3P_4P_5P_6$ is a blue $\ell_6$. Therefore, $X$ is blue, hence the circle $C$ is blue. When $\ell$ is rotated around $AB$, the circle $C$ spans a blue disk with radius $\sqrt{3}$, which contradicts the statement of Lemma~\ref{l:disk}.
\end{proof}

\begin{figure}[H]
\begin{center}
\begin{tikzpicture}[line cap=round,line join=round,>=triangle 45,x=1.0cm,y=1.0cm, scale=1]

\clip(-4.5,-2.5) rectangle (7.5,2.5);
\begin{scope}[xshift=-3.5cm]
\draw [color=black, opacity=0.5] (0,0)--(4, 0);
\draw [color=black, opacity=0.5] (-0.5,0)++(-120:1)--+(5.7, 0);
\draw (-0.9,-0.5) circle (0pt) node{$\ell$};

\draw [fill=red] (0.,0.) node[diamond, scale=0.5, fill=red] {};
\node [above] at (0,0.1) {{$A$}};

\draw [fill=blue] (0,0)++(-120:1) circle (2pt);
\draw [fill=black] (0,-0.3)++(-120:1) circle (0pt) node {$P_1$};

\draw [fill=blue] (0,0)++(-60:1) circle (2pt);
\draw [fill=black] (0,-0.3)++(-60:1) circle (0pt) node {$P_2$};

\draw [fill=blue] (2,0.) circle (2pt);
\node [above] at (2,0.1) {{$X$}};

\draw [fill=white] (2,0)++(-120:1) circle (2pt);
\draw [fill=black] (2,-0.3)++(-120:1) circle (0pt) node {$P_3$};

\draw [fill=white] (2,0)++(-60:1) circle (2pt);
\draw [fill=black] (2,-0.3)++(-60:1) circle (0pt) node {$P_4$};

\draw [fill=red] (4,0.) node[diamond, scale=0.5, fill=red] {};
\node [above] at (4,0.1) {{$B$}};

\draw [fill=blue] (4,0)++(-120:1) circle (2pt);
\draw [fill=blue] (4,-0.3)++(-120:1) circle (0pt) node {$P_5$};

\draw [fill=blue] (4,0)++(-60:1) circle (2pt);
\draw [fill=black] (4,-0.3)++(-60:1) circle (0pt) node {$P_6$};

\draw [color=blue, dashed] (2,0) arc (90:270:0.2 and 0.865);
\draw [color=blue] (2,0) arc (90:-90:0.2 and 0.865);

\end{scope}

\begin{scope}[xshift=2.5cm]
\draw [color=black, opacity=0.5] (0,0)--(4, 0);
\draw [color=black, opacity=0.5] (0,0)++(-120:1)--+(5, 0);

\draw [fill=red] (0.,0.) node[diamond, scale=0.5, fill=red] {};
\node [above] at (0,0.1) {{$A$}};

\draw [fill=blue] (0,0)++(-120:1) circle (2pt);
\draw [fill=black] (0,-0.3)++(-120:1) circle (0pt) node {$P_1$};

\draw [fill=blue] (0,0)++(-60:1) circle (2pt);
\draw [fill=black] (0,-0.3)++(-60:1) circle (0pt) node {$P_2$};

\draw [fill=blue] (2,0.) circle (2pt);

\draw [fill=white] (2,0)++(-120:1) circle (2pt);
\draw [fill=black] (2,-0.3)++(-120:1) circle (0pt) node {$P_3$};

\draw [fill=white] (2,0)++(-60:1) circle (2pt);
\draw [fill=black] (2,-0.3)++(-60:1) circle (0pt) node {$P_4$};

\draw [fill=red] (4,0.) node[diamond, scale=0.5, fill=red] {};
\node [above] at (4,0.1) {{$B$}};

\draw [fill=blue] (4,0)++(-120:1) circle (2pt);
\draw [fill=black] (4,-0.3)++(-120:1) circle (0pt) node {$P_5$};

\draw [fill=blue] (4,0)++(-60:1) circle (2pt);
\draw [fill=black] (4,-0.3)++(-60:1) circle (0pt) node {$P_6$};

\draw [color=blue, dashed] (2,0)++(60:1)++(120:1) arc (90:270:0.3 and 1.732);
\draw [color=blue] (2,0)++(60:1)++(120:1) arc (90:-90:0.3 and 1.732);

\draw[color=white, pattern=north east lines, pattern color=blue] (2,0) circle (0.3 and 1.732);

\draw [color=blue, dashed] (2,1.732) arc (90:270:0.3 and 1.732);
\draw [color=blue] (2,1.732) arc (90:-90:0.3 and 1.732);

\end{scope}

\end{tikzpicture}
\caption{}\label{fig:6}
\end{center}
\end{figure}

\begin{lemma}\label{l:l63}
Let $\mathbb{E}^3$ be coloured in red and blue so that there are no two red points distance $1$ apart. If there are no six blue points forming an $\ell_6$, then there are no two red points distance $3$ apart.
\end{lemma}
\begin{proof}
Assume that points $A$ and $B$ are red and $|AB|=3$. Choose a rectangular coordinate system so that $A$ has coordinates $(1,0,0)$ and $B$ has coordinates $(4,0,0)$. Then, by Lemma~\ref{l:l62}, any point at distance $2$ to $A$ is blue; in particular, the circles $\{(0,y,z):y^2+z^2=3\}$ and $\{(2,y,z):y^2+z^2=3\}$ are blue. By the same argument, circles $\{(3,y,z):y^2+z^2=3\}$ and $\{(5,y,z):y^2+z^2=3\}$ are blue, since their points are at distance $2$ from $B$. 
Consider any line $\ell$ parallel to $AB$ that intersects all the blue circles (for simplicity, let $\ell=\{(x,-\sqrt{3},0): x\in\mathbb{R}\}$), and let $P_1, P_3, P_4, P_6$ be the points of intersection (see Figure~\ref{fig:7}, in this case $P_1(0,-\sqrt{3},0)$, $P_3(2,-\sqrt{3},0)$, $P_4(3,-\sqrt{3},0)$, $P_6(5,-\sqrt{3},0)$). Let $X$ be any point on the circle $C$ with radius $\frac{\sqrt{55}}{2}$ centered at the point on $\ell$ with $x$-coordinate $\frac72$ (on the Figure~\ref{fig:7}, the center has coordinates $(\frac72,-\sqrt{3},0)$). If $X$ is red, then the points $P_2$ and $P_5$ on $\ell$ with $x$-coordinates $1$ and $4$ (on the Figure~\ref{fig:7}, $P_2(1,-\sqrt{3},0)$, $P_5(4,-\sqrt{3},0)$) are both at distance $4$ to $X$. Then,by Lemma~\ref{l:l64}, $P_2$ and $P_5$ are blue, and $P_1P_2P_3P_4P_5P_6$ is a blue $\ell_6$. Therefore, $X$ is blue, hence the circle $C$ is blue. When $\ell$ is rotated around $AB$, the circle $C$ spans the blue annulus bounded by circles of radii $\frac{\sqrt{55}}{2}+\sqrt{3}$ and $\frac{\sqrt{55}}{2}-\sqrt{3}$, which contains a blue disk of radius $\sqrt{3}$. By Lemma~\ref{l:disk}, this leads to a contradiction with the fact that there are no two red points distance $1$ apart.
\end{proof}

\begin{figure}[H]
\begin{center}
\begin{tikzpicture}[line cap=round,line join=round,>=triangle 45,x=1.0cm,y=1.0cm, scale=0.8]
\tikzstyle{dc}   = [circle, minimum width=6pt, draw, inner sep=0pt, path picture={\draw (path picture bounding box.south east) -- (path picture bounding box.north west) (path picture bounding box.south west) -- (path picture bounding box.north east);}]

\clip(-5.5,-6.5) rectangle (7.5,5.5);
\begin{scope}[xshift=-3cm]
\draw [color=black, opacity=0.5] (0,0)--(3, 0);
\draw [color=black, opacity=0.5] (-0.5,0)++(-120:2)--+(5.7, 0);
\draw [color=black, dashed] (0,0)--+(-120:2);
\draw [color=black, dashed] (0,0)--+(-60:2);
\draw [color=black, dashed] (3,0)--+(-120:2);
\draw [color=black, dashed] (3,0)--+(-60:2);
\draw (-1.4,-1.4) circle (0pt) node{$\ell$};

\draw [fill=red] (0.,0.) node[diamond, scale=0.5, fill=red] {};);
\node [above] at (0,0.1) {{$A$}};

\draw [fill=blue] (0,0)++(-120:2) circle (2pt);
\draw [fill=black] (0,-0.3)++(-120:2) circle (0pt) node {$P_1$};

\draw [fill=white] (1,0)++(-120:2) circle (2pt);
\draw [fill=black] (1,-0.3)++(-120:2) circle (0pt) node {$P_2$};

\draw [fill=blue] (1.5,0)++(-120:1)++(-60:1)++(0,-3.708) circle (2pt);
\draw [fill=blue] (1.5,-0.3)++(-120:1)++(-60:1)++(0,-3.708) circle (0pt) node {$X$};

\draw [fill=blue] (0,0)++(-60:2) circle (2pt);
\draw [fill=black] (0,-0.3)++(-60:2) circle (0pt) node {$P_3$};

\draw [fill=blue] (1,0)++(-60:2) circle (2pt);
\draw [fill=black] (1,-0.3)++(-60:2) circle (0pt) node {$P_4$};

\draw [fill=red] (3,0.) node[diamond, scale=0.5, fill=red] {};
\node [above] at (3,0.1) {{$B$}};

\draw [fill=white] (2,0)++(-60:2) circle (2pt);
\draw [fill=blue] (2,-0.3)++(-60:2) circle (0pt) node {$P_5$};

\draw [fill=blue] (3,0)++(-60:2) circle (2pt);
\draw [fill=black] (3,-0.3)++(-60:2) circle (0pt) node {$P_6$};

\draw [color=blue, dashed] (1.5,0)++(-120:1)++(-60:1)++(0,3.708) arc (90:270:0.4 and 3.708);
\draw [color=blue] (1.5,0)++(-120:1)++(-60:1)++(0,3.708) arc (90:-90:0.4 and 3.708);

\end{scope}

\begin{scope}[xshift=3cm]
\draw [color=black, opacity=0.5] (0,0)--(3, 0);
\draw [color=black, opacity=0.5] (0,0)++(-120:2)--+(5, 0);

\draw [fill=red] (0.,0.) node[diamond, scale=0.5, fill=red] {};
\node [above] at (0,0.1) {{$A$}};

\draw [fill=blue] (0,0)++(-120:2) circle (2pt);
\draw [fill=black] (0,-0.3)++(-120:2) circle (0pt) node {$P_1$};

\draw [fill=white] (1,0)++(-120:2) circle (2pt);
\draw [fill=black] (1,-0.3)++(-120:2) circle (0pt) node {$P_2$};

\draw [fill=blue] (1.5,0)++(-120:1)++(-60:1)++(0,-3.708) circle (2pt);
\draw [fill=blue] (1.5,-0.3)++(-120:1)++(-60:1)++(0,-3.708) circle (0pt) node {$X$};

\draw [fill=blue] (0,0)++(-60:2) circle (2pt);
\draw [fill=black] (0,-0.3)++(-60:2) circle (0pt) node {$P_3$};

\draw [fill=blue] (1,0)++(-60:2) circle (2pt);
\draw [fill=black] (1,-0.3)++(-60:2) circle (0pt) node {$P_4$};

\draw [fill=red] (3,0.) node[diamond, scale=0.5, fill=red] {};
\node [above] at (3,0.1) {{$B$}};

\draw [fill=white] (2,0)++(-60:2) circle (2pt);
\draw [fill=blue] (2,-0.3)++(-60:2) circle (0pt) node {$P_5$};

\draw [fill=blue] (3,0)++(-60:2) circle (2pt);
\draw [fill=black] (3,-0.3)++(-60:2) circle (0pt) node {$P_6$};

\draw [color=blue, dashed] (1.5,0)++(0,3.708+1.732) arc (90:270:1 and 3.708+1.732);
\draw [color=blue] (1.5,0)++(0,3.708+1.732) arc (90:-90:1 and 3.708+1.732);

\path [draw=none, pattern=north east lines, pattern color=blue, fill opacity =1, even odd rule] (1.5,0) circle (1 and 3.708+1.732) (1.5,0) circle (0.3 and 3.708-1.732);

\draw [color=blue, dashed] (1.5,0)++(0,3.708-1.732) arc (90:270:0.3 and 3.708-1.732);
\draw [color=blue] (1.5,0)++(0,3.708-1.732) arc (90:-90:0.3 and 3.708-1.732);

\draw [color=blue, dashed] (1.5,0)++(0,3.708+1.732) arc (90:270:0.3 and 1.732);
\draw [color=blue] (1.5,0)++(0,3.708+1.732) arc (90:-90:0.3 and 1.732);

\draw [color=black, dashed](1.5,0)++(0,3.708+1.732)--+(1.5, 0);
\draw [color=black, dashed] (1.5,0)++(0,3.708-1.732)--+(1.5, 0);
\draw [color=black, dashed, <->, thin] (1.5,0)++(0,3.708+1.732)++(1.5, 0)--+(0,-2*1.732);
\draw [fill=blue] (1.5,0)++(0,3.708+1.732)++(1.5, 0)++(0.4,-1.732) circle (0pt) node {$2\sqrt{3}$};

\draw [color=blue, dashed] (1.5,0)++(-120:1)++(-60:1)++(0,3.708) arc (90:270:0.4 and 3.708);
\draw [color=blue] (1.5,0)++(-120:1)++(-60:1)++(0,3.708) arc (90:-90:0.4 and 3.708);

\end{scope}

\end{tikzpicture}
\caption{}\label{fig:7}
\end{center}
\end{figure}

\begin{lemma}\label{r3}
Let $\mathcal{L}$ be a unit triangular lattice on a plane. Let points of $\mathcal{L}$ be coloured in red and blue so that there is no red $\ell_2$ and no blue $\ell_6$. If $\mathcal{L}$ contains two red points distance $\sqrt{3}$ apart, then $\mathcal{L}$ does not contain a blue $\ell_5$.
\end{lemma}
\begin{proof}
Let $A_1$ and $A_2$ be two red nodes of $\mathcal{L}$ distance $\sqrt{3}$ apart. First, it is proved that the node $A_3$ symmetric to $A_1$ in $A_2$ is also red. 
Consider the part of $\mathcal{L}$ depicted in Figure~\ref{fig:8a}.
Points $Q_3$ and $Q_6$ are at distance $3$ from $A_1$, and therefore are blue by Lemma~\ref{l:l63}. Points $Q_4$ and $Q_5$ are at distance $1$ from $A_2$, therefore blue. Then point $P_1$ is blue; otherwise points $Q_1$ and $Q_2$ are both blue and form a blue $\ell_6$ with points $Q_3$, $Q_4$, $Q_5$, $Q_6$.
Points $P_2$ and $P_6$ are at distance $4$ from $A_1$, therefore are blue by Lemma~\ref{l:l64}. Points $P_3$ and $P_5$ are at distance $2$ from $A_2$, therefore are blue by Lemma~\ref{l:l62}. Then the point $A_3$ has to be red in order to prevent the blue $P_1P_2P_3A_3P_5P_6$.\\
Using the same argument, it can be proved that the node $A_4$ symmetric to $A_2$ in $A_3$ is red, and similarly for any $k\in\mathbb{Z}$ a point on the line $A_1A_2$ at distance $k\sqrt{3}$ from $A_1$ is red.\\
Consider five consecutive red nodes $A_1$, $A_2$, $A_3$, $A_4$, $A_5$ on the line $A_1A_2$ and the part of the lattice depicted in Figure~\ref{fig:8b}. By Lemmas \ref{l:l62}, \ref{l:l63}, \ref{l:l64}, the points $P_1$, $P_2$, $P_3$, $P_4$ are blue, since they are at distance $1$, $2$, $3$ and $4$ from $A_1$, respectively. By Lemma~\ref{l:l63}, the point $P_6$ is blue (since it is $3$ apart from $A_4$). Then point $P_5$ is red (otherwise $P_1$, $P_2$, $P_3$, $P_4$, $P_6$ form a blue $\ell_6$). Similarly, the points $Q_1$, $Q_2$, $Q_3$, $Q_6$ are blue, hence $Q_5$ is red. Then $P_5$ and $Q_5$ are two red nodes distance $\sqrt{3}$ apart, which forces all the nodes of $\mathcal{L}$ on the line $P_5Q_5$ to be red.\\
\begin{figure}[H]
    \begin{subfigure}{0.5\linewidth}
        
\begin{tikzpicture}[line cap=round,line join=round,>=triangle 45,x=1.0cm,y=1.0cm, scale=0.8]

\clip(-4,-1) rectangle (4,5.5);

\foreach \a in {1,2,...,3}
\draw [dashed, opacity=0.5] (-4,0)++(60:\a-1)++(120:\a-1)--+(6,0);

\foreach \a in {1,2}
\draw [dashed, opacity=0.5] (-4.5,0)++(60:\a)++(120:\a-1)--+(6,0);

\foreach \a in {1,2,...,6}
\draw [dashed, opacity=0.5] (-6+\a,0)--+(60:4);

\foreach \a in {1,2,...,6}
\draw [dashed, opacity=0.5] (3-\a,0)--+(120:4);

\draw [dashed, opacity=0.5] (1,0)--+(60:2);

\draw [dashed, opacity=0.5] (1,0)++(60:2)++(120:2)--+(-60:2);

\draw [fill=red,  opacity=1] (0,0) node[diamond, scale=0.7, fill=red] {};
\draw [fill=red] (0,0)++(0,-0.4) circle (0pt) node {$A_1$};

\draw [fill=red] (0,0)++(60:1)++(120:1) node[diamond, scale=0.7, fill=red] {};
\draw [fill=red] (0,0)++(0,-0.4)++(60:1)++(120:1) circle (0pt) node {$A_2$};

\foreach \a in {3,4,5,6}
\draw [fill=blue] (0,0)++(120:3)++(\a-3,0) circle (2pt);

\foreach \a in {3,4,5,6}
\draw [fill=red] (0,0)++(0,-0.4)++(120:3)++(\a-3,0) circle (0pt) node {$Q_\a$};

\draw [fill=white] (0,0)++(120:3)++(-1,0) circle (2pt); 
\draw [fill=red] (0,0)++(0,-0.4)++(120:3)++(-1,0) circle (0pt) node {$Q_2$};

\draw [fill=white] (0,0)++(120:3)++(-2,0) circle (2pt);
\draw [fill=red] (0,0)++(0,-0.4)++(120:3)++(-2,0) circle (0pt) node {$Q_1$};

\draw [fill=blue] (0,0)++(120:4)++(-1,0) circle (2pt);
\draw [fill=red] (0,0)++(0,-0.4)++(120:4)++(-1,0) circle (0pt) node {$P_1$};

\draw [fill=blue] (0,0)++(120:4)++(0,0) circle (2pt);
\draw [fill=red] (0,0)++(0,-0.4)++(120:4)++(0,0) circle (0pt) node {$P_2$};

\draw [fill=blue] (0,0)++(120:4)++(1,0) circle (2pt);
\draw [fill=red] (0,0)++(0,-0.4)++(120:4)++(1,0) circle (0pt) node {$P_3$};

\draw [fill=white] (0,0)++(120:4)++(2,0) circle (2pt);
\draw [fill=red] (0,0)++(0,-0.4)++(120:4)++(2,0) circle (0pt) node {$A_3$};

\draw [fill=blue] (0,0)++(120:4)++(3,0) circle (2pt);
\draw [fill=red] (0,0)++(0,-0.4)++(120:4)++(3,0) circle (0pt) node {$P_5$};

\draw [fill=blue] (0,0)++(120:4)++(4,0) circle (2pt);
\draw [fill=red] (0,0)++(0,-0.4)++(120:4)++(4,0) circle (0pt) node {$P_6$};

\end{tikzpicture}
        \caption{}
        \label{fig:8a}
    \end{subfigure}%
    \begin{subfigure}{0.5\linewidth}
        \centering
\begin{tikzpicture}[line cap=round,line join=round,>=triangle 45,x=1.0cm,y=1.0cm, scale=0.8]
\clip(-0.5,-0.8) rectangle (4,7.5);

\foreach \a in {1,...,10}
\draw [dashed, opacity=0.5] (-2,0)++(60:\a-2)++(120:\a-2)--+(10,0);

\foreach \a in {1,...,20}
\draw [dashed, opacity=0.5] (-10,0)++(\a,0)--+(60:10);

\foreach \a in {1,...,20}
\draw [dashed, opacity=0.5] (-10,0)++(\a,0)--+(120:10);

\foreach \a in {1,2,...,10}
\draw [dashed, opacity=0.5] (-2,0)++(60:\a-3)++(120:\a-4)--+(10,0);

\foreach \a in {1,2,...,5}
\draw [fill=red] (0,0)++(60:\a-1)++(120:\a-1) node[diamond, scale=0.7, fill=red] {};

\foreach \a in {1,2,...,5}
\draw [fill=red] (0,0)++(0,-0.4)++(60:\a-1)++(120:\a-1) circle (0pt) node {$A_\a$};

\foreach \a in {1,2,...,6}
\draw [fill=blue] (0,0)++(60:\a) circle (2pt);

\foreach \a in {1,2,...,6}
\draw [fill=blue] (0,0)++(0,-0.4)++(60:\a) circle (0pt) node {$P_\a$};

\foreach \a in {1,2,...,6}
\draw [fill=blue] (0,0)++(60:4)++(120:4)++(-60:\a) circle (2pt);

\foreach \a in {1,2,...,3}
\draw [fill=blue] (0,-0.4)++(60:4)++(120:4)++(-60:\a) circle (0pt) node {$Q_\a$};

\foreach \a in {5,6}
\draw [fill=blue] (0,-0.4)++(60:4)++(120:4)++(-60:\a) circle (0pt) node {$Q_\a$};

\draw [fill=red] (0,0)++(60:4)++(120:4)++(-60:5) node[diamond, scale=0.7, fill=red] {};

\draw [fill=red] (0,0)++(60:5) node[diamond, scale=0.7, fill=red] {};

\end{tikzpicture}
        \caption{}
        \label{fig:8b}
    \end{subfigure}
    \caption{} 
    \label{fig:8}
\end{figure}

If the colouring is expanded further in a similar way, the plane will be coloured in the unique way, which is shown in Figure~\ref{fig:10}. Every $\ell_5$ that belongs to $\mathcal{L}$ contains a red point, therefore the colouring does not contain a blue $\ell_5$.
\end{proof}

\begin{figure}[H]
\begin{center}
\begin{tikzpicture}[line cap=round,line join=round,>=triangle 45,x=1.0cm,y=1.0cm,scale=0.8]
\clip(-2.9,0.3) rectangle (12.7,6.5);

\foreach \a in {0,1,...,10}
\draw [dashed, opacity=0.5] (-10,0)++(60:\a)++(120:\a)--+(32,0);

\foreach \a in {0,1,...,10}
\draw [dashed, opacity=0.5] (-10,0)++(60:\a+1)++(120:\a)--+(32,0);

\foreach \a in {0,1,...,30}
\draw [dashed, opacity=0.5] (-10,0)++(\a,0)--+(60:10);

\foreach \a in {0,1,...,30}
\draw [dashed, opacity=0.5] (-10,0)++(\a,0)--+(120:10);

\foreach \b in {0,1,...,20}
\foreach \a in {0,1,...,20}
\draw [fill=blue] (-5,0)++(\a,0.)++(60:\b)++(120:\b) circle (2pt);

\foreach \b in {0,1,...,20}
\foreach \a in {0,1,...,20}
\draw [fill=blue] (-5,0)++(\a,0)++(60:1)++(60:\b)++(120:\b) circle (2pt);

\foreach \a in {0,...,5}
\draw [fill=red] (0,0)++(60:\a)++(120:\a) node[diamond, scale=0.7, fill=red] {};

\foreach \a in {0,...,4}
\draw [fill=red] (2,0)++(60:\a+1)++(120:\a) node[diamond, scale=0.7, fill=red] {};

\foreach \a in {0,...,5}
\draw [fill=red] (5,0)++(60:\a)++(120:\a) node[diamond, scale=0.7, fill=red] {};

\foreach \a in {0,...,4}
\draw [fill=red] (7,0)++(60:\a+1)++(120:\a) node[diamond, scale=0.7, fill=red] {};

\foreach \a in {0,...,4}
\draw [fill=red] (12,0)++(60:\a+1)++(120:\a) node[diamond, scale=0.7, fill=red] {};

\foreach \a in {0,...,4}
\draw [fill=red] (-3,0)++(60:\a+1)++(120:\a) node[diamond, scale=0.7, fill=red] {};

\foreach \a in {0,...,5}
\draw [fill=red] (10,0)++(60:\a)++(120:\a) node[diamond, scale=0.7, fill=red] {};


\end{tikzpicture}
\end{center}
\caption{}\label{fig:10}
\end{figure}

\begin{proof}[Proof of the Theorem~\ref{thm:l6}]
Suppose that $\mathbb{E}^3$ is coloured in two colours so that there is no red $\ell_2$ and no blue $\ell_6$. By Theorem~\ref{thm:l5}, there is a blue $\ell_5$, say, $X_1X_2X_3X_4X_5$. Consider any unit triangular lattice $\mathcal{L}$ such that $X_1$, $X_2$, $X_3$, $X_4$, $X_5$ are nodes of $\mathcal{L}$. Since $\mathcal{L}$ does not contain a blue $\ell_6$, there is a red node $A$ in $\mathcal{L}$. Consider the part of $\mathcal{L}$ depicted in Figure~\ref{fig:9}. The points $P_2$ and $P_3$ are blue, since they are distance $1$ apart from $A$. Since $\mathcal{L}$ contains a blue $\ell_5$, by Lemma~\ref{r3}, there are no two red nodes of $\mathcal{L}$ distance $\sqrt{3}$ apart, therefore points $P_1$ and $P_4$ are blue. Then points $B$ and $C$ can not be both blue (otherwise a blue $\ell_6$ is formed), therefore one of them, say, $B$, is red. Then points $Q_5$ and $Q_6$ are at distance $1$ and $\sqrt{3}$ from $B$, hence blue. The points $Q_3$, $Q_2$, $Q_1$ are distance $1$, $1$, $\sqrt{3}$ apart from $A$, respectively, therefore blue. Hence, the points $Q_1$, $Q_2$, $Q_3$, $P_1$, $Q_5$, $Q_6$ form a blue $\ell_6$, which contradicts the initial assumption.
\end{proof}

\begin{figure}[H]
\begin{center}
\begin{tikzpicture}[line cap=round,line join=round,>=triangle 45,x=1.0cm,y=1.0cm]
\clip(-3,-3.2) rectangle (3,1.85);

\foreach \a in {1,2,...,10}
\draw [dashed, opacity=0.5] (-6,0)++(-60:1)++(-120:2)++(\a-1,0)--+(60:5);

\foreach \a in {1,2,...,10}
\draw [dashed, opacity=0.5] (6,0)++(-60:1)++(-120:2)++(-\a+1,0)--+(120:5);

\foreach \a in {1,2,...,3}
\draw [dashed, opacity=0.5] (0,0)++(-60:2-\a)++(-120:2-\a)++(-4,0)--+(7,0);

\foreach \a in {1,2,...,3}
\draw [dashed, opacity=0.5] (0,0)++(-60:2-\a+1)++(-120:2-\a)++(-4,0)--+(7,0);

\draw [fill=red] (0,0) node[diamond, scale=0.7, fill=red] {};
\draw [fill=red] (0,0)++(0,-0.4) circle (0pt) node {$A$};

\foreach \a in {1,2,3,4}
\draw [fill=blue] (0,0)++(-120:1)++(\a-2,0) circle (2pt);

\foreach \a in {1,2,3,4}
\draw [fill=blue] (0,-0.4)++(-120:1)++(\a-2,0) circle (0pt) node {$P_\a$};

\foreach \a in {1,2,3}
\draw [fill=blue] (0,0)++(120:1)++(60:1)++(-120:\a-1) circle (2pt);

\foreach \a in {1,2,3}
\draw [fill=blue] (0,-0.4)++(120:1)++(60:1)++(-120:\a-1) circle (0pt) node {$Q_\a$};

\foreach \a in {5,6}
\draw [fill=white] (0,0)++(120:1)++(60:1)++(-120:\a-1) circle (2pt);

\foreach \a in {5,6}
\draw [fill=blue] (0,-0.4)++(120:1)++(60:1)++(-120:\a-1) circle (0pt) node {$Q_\a$};

\foreach \a in {1,2,3}
\draw [fill=blue] (0,0)++(120:1)++(60:1)++(-60:\a-1) circle (2pt);

\draw [fill=white] (0,0)++(2,0)++(-60:1) circle (2pt);
\draw [fill=white] (0,-0.4)++(2,0)++(-60:1) circle (0pt) node {$C$};

\draw [fill=red] (0,0)++(-3,0)++(-60:1) node[diamond, scale=0.7, fill=red] {};
\draw [fill=red] (0,-0.4)++(-3,0)++(-60:1) circle (0pt) node {$B$};

\end{tikzpicture}
\end{center}
\caption{}\label{fig:9}
\end{figure}

\section{Acknowledgements}
The authors would like to thank Ron Graham and Roz\'alia Juh\'asz for providing information about the current state of the problem.


\begin{bibdiv}
\begin{biblist}[\normalsize]

\bib{erdos1}{article}{
   author={Erd\H os, P.},
   author={Graham, R. L.},
   author={Montgomery, Paul},
   author={Rothschild, B. L.},
   author={Spencer, J.},
   author={Straus, E. G.},
   title={\emph{Euclidean Ramsey theorems. I}},
   journal={\emph{J. Combin. Theory Ser. A}},
   volume={14},
   date={1973},
   pages={341--363},
}

\bib{erdos2}{article}{
   author={Erd\H os, Paul},
   author={Graham, R. L.},
   author={Montgomery, P.},
   author={Rothschild, B. L.},
   author={Spencer, J.},
   author={Straus, E. G.},
   title={\emph{Euclidean Ramsey theorems. II}},
   conference={
      title={\emph{Infinite and finite sets (Colloq., Keszthely, 1973; dedicated
      to P. Erd\H os on his 60th birthday), Vol. I}},
   },
   book={
      publisher={North-Holland, Amsterdam},
   },
   date={1975},
   pages={529--557. Colloq. Math. Soc. J\'anos Bolyai, Vol. 10},
}

\bib{erdos3}{article}{
   author={Erd\H os, P.},
   author={Graham, R. L.},
   author={Montgomery, P.},
   author={Rothschild, B. L.},
   author={Spencer, J.},
   author={Straus, E. G.},
   title={\emph{Euclidean Ramsey theorems. III}},
   conference={
      title={\emph{Infinite and finite sets (Colloq., Keszthely, 1973; dedicated
      to P. Erd\H os on his 60th birthday), Vol. I}},
   },
   book={
      publisher={North-Holland, Amsterdam},
   },
   date={1975},
   pages={559--583. Colloq. Math. Soc. J\'anos Bolyai, Vol. 10},
}

\bib{graham}{book}{
   title={\emph{Euclidean Ramsey theory, in} Handbook of discrete and computational geometry \emph{(J. E. Goodman and J. O'Rourke, Eds.)}},
   edition={2},
   author={R. L. Graham},
   editor={},
   publisher={Chapman \& Hall/CRC, Boca Raton, FL},
   date={2004},
}

\bib{ivan}{thesis}{
   author={Iv\'an, L\'aszl\'o},
   title={Monochromatic point sets in the plane and in the space},
   date={1979},
   note={Master’s Thesis, University of Szeged, Bolyai Institute (in Hungarian)}
}



\bib{juhasz}{article}{
   author={Juh\'asz, Roz\'alia},
   title={\emph{Ramsey type theorems in the plane}},
   journal={\emph{J. Combin. Theory Ser. A}},
   volume={27},
   date={1979},
   pages={152--160},
}



\end{biblist}
\end{bibdiv}
\end{document}